\documentclass[a4paper]{amsart}
\pdfoutput = 1

\usepackage[numbers]{natbib}
\usepackage{amsmath, amssymb,stmaryrd}% ,mathdots?
\usepackage{type1cm}
\usepackage{tensor}
\usepackage{enumerate}
\usepackage{graphicx}
\usepackage[czech,english]{babel}
\usepackage[single]{accents}
\usepackage{mdwtab}
\usepackage[OT1]{fontenc}
\usepackage[all]{xy}
\usepackage{mathtools}

 \newcommand{\rdens}[1]{\nabla_{{#1}}}
 \newcommand{\trdens}[1]{\widetilde{\nabla}_{{#1}}}
 \newcommand{\irdens}[1]{\accentset{\propto}{\nabla}_{#1}}
\newcommand{\shtm}{\,\widetilde{\triangledown}\,}
\newcommand{\shm}{\,{\triangledown}\,}
\newcommand{\shim}{\,\accentset{\propto}{\triangledown}\,}

\newcommand{\bbbn}{\mathbb{N}}

\usepackage{color}
\definecolor{shadecolor}{gray}{.85}%
\definecolor{tintedcolor}{gray}{.80}%
\usepackage{framed}%
\definecolor{mytintedcolor}{gray}{.95}%
\makeatletter
\newdimen\svparindent
\newcounter{tmpthm}
\newenvironment{mytinted}{%
  \MakeFramed {\FrameRestore}}%
{\endMakeFramed}
{\endlist\end{mytinted}\egroup}
\makeatother
\newlength{\graphshift}
\newcommand{\nlongrightarrow}{\relbar\!\joinrel\not\relbar\joinrel\!\!\rightarrow}
\newcommand{\duality}[3][1cm]{\setlength{\graphshift}{#1}\raisebox{-.4\graphshift}{\includegraphics[height=\graphshift]{#2}}\quad\nlongrightarrow\quad G\qquad\iff\qquad  G\quad\longrightarrow \quad\raisebox{-.4\graphshift}{\includegraphics[height=\graphshift]{#3}}}

\newtheorem{theorem}{Theorem}

\newtheorem{corollary}{Corollary}
\newtheorem{lemma}[theorem]{Lemma}

\theoremstyle{definition}
\newtheorem{definition}[theorem]{Definition}
\newtheorem{example}[theorem]{Example}

\theoremstyle{remark}

\newtheorem{problem}{Problem}

\begin{document}

\title{On Low Tree-Depth Decompositions}
\author{Jaroslav Ne{\v s}et{\v r}il}
\address{Jaroslav Ne{\v s}et{\v r}il\\
Computer Science Institute of Charles University (IUUK and ITI)\\
   Malostransk\' e n\' am.25, 11800 Praha 1, Czech Republic}
\email{nesetril@iuuk.mff.cuni.cz}
\thanks{Supported by grant ERCCZ LL-1201
and CE-ITI P202/12/G061, and by the European Associated Laboratory ``Structures in
Combinatorics'' (LEA STRUCO)}

\author{Patrice Ossona~de~Mendez}
\address{Patrice~Ossona~de~Mendez\\
Centre d'Analyse et de Math\'ematiques Sociales (CNRS, UMR 8557)\\
  190-198 avenue de France, 75013 Paris, France
	--- and ---
Computer Science Institute of Charles University (IUUK)\\
   Malostransk\' e n\' am.25, 11800 Praha 1, Czech Republic}
\email{pom@ehess.fr}
\thanks{Supported by grant ERCCZ LL-1201 and by the European Associated Laboratory ``Structures in
Combinatorics'' (LEA STRUCO), and partially supported by ANR project Stint under reference ANR-13-BS02-0007}

\date{\today}
\begin{abstract}
The theory of sparse structures usually uses tree like structures as building blocks. In the context of sparse/dense dichotomy this role is played by graphs with bounded tree depth. In this paper we survey results related to this concept and particularly explain how these graphs are used to decompose and construct more complex graphs and structures. In more technical terms we survey some of the properties and applications of low tree depth decomposition of graphs.
\end{abstract}
\maketitle
\section{Tree-Depth}
The {\em tree-depth} of a graph is a minor montone graph
invariant that has been defined in \cite{Taxi_tdepth}, and which is equivalent or
similar to  the {\em rank
function} (used for the analysis of countable graphs, see e.g.\  \cite{Nev2003}),
the {\em vertex ranking number} \cite{vertex_ranking,Schaffer}, and the minimum height of an elimination tree \cite{Bodlaender1995}.
Tree-depth can also be seen as an analog for undirected graphs of the cycle
rank defined by Eggan \cite{Eggan1963}, which is  a parameter relating
digraph complexity to other areas such as regular language complexity and
asymmetric matrix factorization.
The notion of tree-depth found a wide range of applications, from the study of non-repetitive coloring \cite{Thue_choos} to the proof of the homomorphism preservation theorem for finite structures \cite{Rossman2007}. Recall the definition of tree-depth:

\begin{definition}
The {\em tree-depth} ${\rm td}(G)$ of a graph $G$ is
defined as
the minimum height\footnote{Here the height is defined as the maximum number of vertices in a chain from a root to a leaf} of a rooted forest $Y$ such that $G$ is a subgraph of the closure of $Y$ (that is of the graph obtained by adding edges between a vertex and all its ancestors). In particular, the tree-depth of a disconnected graph is the maximum of the tree-depths of its connected components.
\end{definition}

Several characterizations of tree-depth have been given, which can be seen as possible alternative definitions. Let us mention:

\begin{trivlist}
 \setlength{\itemsep}{3mm}
\item {\bf TD$1$.} The tree-depth of a graph is the order of the largest clique in a trivially perfect supergraph of $G$ \cite{td-wiki}. Recall that a graph is {\em trivially perfect} if it has the property that in each of its induced subgraphs the size of the maximum independent set equals the number of maximal cliques \cite{Golumbic1978105}. This characterization follows directly from the property that a connected graph is trivially perfect if and only if it is the comparability graph of a rooted tree \cite{Golumbic1978105}.

\item {\bf TD$2$.} The tree-depth of a graph is the minimum number of colors in a {\em centered coloring} of $G$, that is in a vertex coloring of $G$ such that in every connected subgraph of $G$ some color appears exactly once \cite{Taxi_tdepth}.
\item {\bf TD$3$.}
A strongly related notion is vertex ranking, which has has been investigated in \cite{vertex_ranking,Schaffer}.
The {\em vertex ranking} (or {\em ordered coloring}) of a graph is a vertex coloring by a linearly ordered set of colors such that for every path in the graph with end vertices of the same color there is a vertex on this path with a higher color. The equality of the minimum number %of colors in a centered coloring, the minimum number
of colors in a vertex ranking and the tree-detph is proved in \cite{Taxi_tdepth}.

\item {\bf TD$4$.} The tree-depth of a graph $G$ with connected components $G_1,\dots,G_p$, is recursively defined by:
$$
{\rm td}(G)=\begin{cases}
1&\text{ if }G\simeq K_1\\
\displaystyle\max_{i=1}^p {\rm td}(G_i)&\text{ if $G$ is disconnected}\\
\displaystyle 1+\min_{v\in V(G)}{\rm td}(G-v)&\text{ if $G$ is connected and }G\not\simeq K_1
\end{cases}
$$
The equivalence between the value given by this recursive definition and minimum height of an elimination tree, as well as the equality of this value with the tree-depth are proved in \cite{Taxi_tdepth}.
\item {\bf TD$5$.} The tree-depth can also be defined by means of games, see \cite{Giannopoulou2011, Gruber2008,Hunter2011}.
In particular, this leads to a min-max formula for tree-depth in the spirit of the min-max formula relating tree-width and bramble size \cite{Seymour1993}. Precisely, a {\em shelter} in a graph $G$ is a family $\mathcal S$ of non-empty connected
subgraphs of $G$ partially ordered by
inclusion such that for every subgraph $H\in\mathcal S$ not minimal in $\mathcal F$ and for every $x\in H$ there exists $H'\in\mathcal S$ covered by $H$ (in the partial order) such that $x\not\in H'$.
The {\em thickness} of a shelter $\mathcal S$ is the minimal length of a maximal chain of $\mathcal S$. Then  the tree-depth of a graph $G$ equals the maximum thickness of a shelter in $G$ \cite{Giannopoulou2011}.
\item {\bf TD$6$.} Also, graphs with tree-depth at most $t$ can be theoretically characterized by means of a finite set of forbidden minors, subgraphs, or even induced subgraphs. But in each case, the number of obstruction grows at least like a double (and at most a triple) exponential in $t$ \cite{Dvorak2012969}.
\end{trivlist}

More generally, classes with bounded tree-depth can be characterized by several properties:

\begin{trivlist}
 \setlength{\itemsep}{3mm}
\item {\bf TD$7$.} A class of graphs $\mathcal C$ has bounded tree-depth if and only if there is some integer $k$ such that graphs in $\mathcal C$ exclude  $P_k$ as a subgraph.
More precisely, while computing the tree-depth of a graph $G$ is a hard problem, it can be (very roughly) approximated bu considering the height $h$ of a Depth-First Search tree of $G$, as $\lceil\log_2 (h+2)\rceil\leq {\rm td}(G)\leq h$ \cite{Sparsity}.
\item {\bf TD$8$.} A class of graphs $\mathcal C$ has bounded tree-depth if and only if there is some integers $s,t,q$ such that graphs in
$\mathcal C$ exclude $P_s, K_t,$ and $K_{q,q}$ as induced subgraphs (this follows from the previous item and \cite[Theorem 3]{Atminas2014}, which states that for every $s$, $t$, and $q$, there is a number $Z = Z(s, t, q)$ such that every graph with a path of length at least $Z$ contains either $P_s$ or $K_t$ or $K_{q,q}$ as an induced subgraph.
\item {\bf TD$9$.}
A monotone class of graphs has bounded tree-depth if and only if
it is  well quasi-ordered for the induced-subgraph relation (with vertices possibly colored using $k\geq 2$ colors)  (follows from
\cite{Ding1992}).
\item {\bf TD$10$.}
A monotone class of graphs has bounded tree-depth if and only if
First-order logic (FO) and monadic second-order (MSO) logic have the same expressive power on the class \cite{6280445}.
\end{trivlist}

Classes of graphs with tree-depth at most $t$ are computationally very simple, as witnessed by the following properties:

\begin{trivlist}
 \setlength{\itemsep}{3mm}
\item It follows from {\bf TD$9$} that  every hereditary property can be tested in polynomial time when restricted to graphs with tree-depth at most $t$.
Let us emphasize how one can combine {\bf TD$8$} and {\bf TD$9$}
to get complexity results for $P_s$-free graphs. Recall that a graph $G$ is {\em $k$-choosable}  if for every assignment of a set $S(v)$ of $k$ colors to every vertex $v$ of $G$, there is a proper coloring of $G$ that assigns to each vertex $v$ a color from $S(v)$ \cite{vizing,ert}. Note that in general, for $k>2$, deciding $k$-choosability for bipartite graphs is $\Pi_2^P$-complete, hence more difficult that both NP and co-NP problems. It was proved in \cite{Heggernes2009} that for $P_5$-free graphs, that is, graphs excluding  $P_5$ as an induced subgraph, k-choosability is fixed-parameter tractable. For general $P_s$-free graphs we prove:
\begin{theorem}
For every integers $s$ and  $k$, there is a polynomial time algorithm to decide whether a $P_s$-free graph $G$  is $k$-choosable.
\end{theorem}
\begin{proof}
Assume $G$ is $P_s$-free. We can decide in polynomial time whether $G$ includes $K_{k+1}$ or $K_{k,k^k}$ as an induced subgraph.
In the affirmative, $G$ is not $k$-choosable. Otherwise, the tree-depth of $G$ is bounded by some constant $C(s,k)$. As the property to be $k$-choosable is hereditary, we can use a polynomial time algorithm  deciding whether a graph with tree-depth at most $C(s,k)$
is $k$-choosable
\end{proof}

\item Graphs with tree-depth at most $t$ have a (homomorphism) core of order bounded by a function of $t$ \cite{Taxi_tdepth}. In other word, every graph $G$ with tree-depth at most $t$ has an induced subgraph $H$ of order at most $F(t)$ such that there exists an adjacency preserving map (that is: a {\em homomorphism}) from $V(G)$ to $V(H)$.
\item The complexity of checking the satisfaction of an ${\rm MSO}_2$ property $\phi$ on a class with tree-depth at most $t$ in time $O(f(\phi, t)\cdot|G|)$, where $f$ has an elementary dependence on $\phi$ \cite{Gajarsky2012}. This is in contrast with the dependence arising for ${\rm MSO}_2$-model checking in classes with bounded treewidth using Courcelle's algorithm \cite{Courcelle2}, where $f$ involves a tower of exponents of height growing with $\phi$ (what is generally unavoidable \cite{Frick2004}). These properties led to the study of classes with bounded shrub-depth, generalizing classes with bounded tree-depth, and enjoying similar properties for ${\rm MSO}_1$-logic
\cite{Gajarsky2012,Ganian2012}. Concerning the dependency on the tree-depth $t$, note that the
$(t + 1)$-fold exponential algorithm for MSO model-checking given by
Gajarsk{\'y} and Hlin{\v e}n\'y  in \cite{gajarsky2012faster} is essentially optimal \cite{Lampis2013}.
\end{trivlist}

Graphs with bounded tree depth form the building blocs for more complicated graphs, with which we deal in the next section.
\section{Low Tree-Depth Decomposition of Graphs}
Several extensions of chromatic number of been proposed and studied in the literature. For instance, the {\em acyclic chromatic number} is the minimum number of colors in a proper vertex-coloring such that any two colors induce an acyclic graph (see e.g. \cite{AMS,Borodin1979}). More generally, for a fixed parameter $p$, one can ask what is the minimum number of colors in a proper vertex-coloring of a graph $G$, such that any subset $I$ of at most $p$ colors induce a subgraph with treewidth at most $|I|-1$. In this setting, the value obtained for $p=1$ is the chromatic number, while the value obtained for $p=2$ is the acyclic chromatic number.

In this setting, the following result  has been proved by Devos, Oporowski, Sanders, Reed, Seymour and Vertigan using the structure theorem for graphs excluding a minor:

\begin{theorem}[\cite{2tw}]
\label{th:2tw}For every proper minor closed class $\mathcal K$  and
integer $k\geq 1$, there is an integer $N=N({\mathcal K},k)$,
such that every graph $G \in
{\mathcal K}$ has a vertex partition into $N$ graphs such that
any $j\leq k$ parts form a graph with tree-width at most $j-1$.
\end{theorem}

The stronger concept of low tree-depth decomposition has been introduced by the authors in \cite{Taxi_tdepth}.

\begin{definition}
A {\em low tree-depth decomposition} with parameter $p$ of a graph $G$ is a coloring of the vertices of $G$, such that any subset $I$ of at most $p$ colors induce a subgraph with tree-depth at least $|I|$. The minimum number of colors in a low tree-depth decomposition with parameter $p$ of $G$ is denoted by $\chi_p(G)$.
\end{definition}

For instance, $\chi_1(G)$ is the (standard) chromatic number of $G$, while $\chi_2(G)$ is the {\em star chromatic number} of $G$, that is the minimum number of colors in a proper vertex-coloring of $G$ such that any two colors induce a star forest (see e.g. \cite{alon2,Taxi_jcolor}).

The authors were able to extend Theorem~\ref{th:2tw} to low tree-depth decomposition in \cite{Taxi_tdepth}.
Then, using the concept of transitive fraternal augmentation \cite{POMNI}, the authors extended further existence of low tree-depth decomposition (with bounded number of colors) to classes with bounded expansion, the definition of which we recall now:
\begin{definition}
A class $\mathcal C$ has {\em bounded expansion} if there exists a function $f:\bbbn\rightarrow\bbbn$ such that every topological minor $H$ of a graph $G\in\mathcal{C}$ has an average degree bounded by $f(p)$, where $p$ is the maximum number of subdivisions per edge needed to turn $H$ into a subgraph of $G$.
\end{definition}

Extending low tree-depth decomposition to classes with bounded expansion in  is the best possible:

\begin{theorem}[\cite{POMNI}]
\label{thm:chiBE}
Let $\mathcal{C}$ be a class of graphs, then the following are equivalent:
\begin{enumerate}
\item for every integer $p$ it holds $\sup_{G\in\mathcal{C}}\chi_p(G)<\infty$;

\item the class $\mathcal{C}$ has bounded expansion.
\end{enumerate}
\end{theorem}

Properties and characterizations of classes with bounded expansion will be discussed in more details in Section~\ref{sec:taxonomy} (we refer the reader to \cite{Sparsity} for a thorough analysis). Let us mention that classes with bounded expansion in particular include proper minor closed classes (as for instance planar graphs or graphs embeddable on some fixed surface), classes with bounded degree, and more generally classes excluding a topological minor. Thus on the one side the classes of graphs with bounded expansion include most of the sparse classes of structural graph theory, yet on the other side they have pleasant algorithmic and extremal properties.

On the other hand, one could ask whether for proper minor-closed classes one could ask  there exists a stronger coloring than the  one given by low tree-depth decompositions. Precisely, one can ask what is
 the minimum number of colors required for a vertex coloring of a graph $G$, so that any subgraph $H$ of $G$ gets at least $f (H)$ colors. (For instance that the star coloring corresponds to the graph function where any $P_4$ gets at least $3$ colors.)
Define the {\em upper chromatic number} $\overline{\chi}(H)$ of a graph $H$ as the greatest integer, such that for any proper minor closed class of graph $\mathcal C$, there exists a constant $N=N(\mathcal C, H)$, such that any graph $G\in\mathcal C$ has a vertex coloring by at $N$ colors so that any subgraph of $G$ isomorphic to $H$ gets at least $\overline{\chi}(H)$ colors.
The authors proved in \cite{Taxi_tdepth} that $\overline{\chi}(H)={\rm td}(H)$, showing that low tree-depth decomposition is the best we can achieve for proper minor closed classes. Note that the tree-depth of a graph $G$ is also related to the chromatic numbers $\chi_p(G)$ by ${\rm td}(G)=\max_p \chi_p(G)$ \cite{Taxi_tdepth}.

\section{Low Tree-Depth Decomposition and Restricted Dualities}

 The original motivation of low tree-depth decomposition was to prove the existence
of a triangle free graph $H$ such that every triangle-free planar $G$ admits a homomorphism to $H$, thus providing a structural strengthening of Gr\" otzsch's theorem \cite{Taxi_jcolor}.
Recall that a {\em homomorphism} of a graph $G$ to a graph $H$ is a mapping from the vertex set $V(G)$ of $G$ to the vertex set $V(H)$ of $H$ that preserves adjacency.
The existence (resp. non-existence) of a homomorphism of $G$ to $H$ will be denoted by $G\rightarrow H$ (resp. by $G\nrightarrow H$).
We refer the interested reader to the monograph \cite{HN} for a detailed study of graph homomorphisms.

Thus the above planar triangle-free problem can be restated as follows: Prove that there exists
a graph $H$ such that $K_3\nrightarrow H$ and such that for every planar graph $G$ it holds
$$
K_3\nrightarrow G\quad\iff\quad G\rightarrow H.
$$
More generally, we are interested in the following problem: given a class of graphs $\mathcal C$ and a connected graph $F$, find a graph $D_{\mathcal{C}}(F)$ for $\mathcal{C}$ (which we shall refer to as a {\em dual} of $F$ for $\mathcal C$), such that $F\nrightarrow D_{\mathcal{C}}(F)$ and such that for every $G\in\mathcal C$ it holds
$$
F\nrightarrow G\quad\iff\quad G\rightarrow D_{\mathcal{C}}(F).
$$
(Note that $D_{\mathcal{C}}(F)$ is not uniquely determined by the above equivalence.)
A couple $(F,D_{\mathcal{C}}(F))$ with the above property is called
a {\em restricted duality} of $\mathcal C$.

\begin{example}
For the special case of triangle-free planar graphs,  the existence of a dual was proved by the authors in \cite{Taxi_tdepth} and
the minimum order dual has been proved to be the Clebsch graph
by Naserasr \cite{Nas}.

$\forall\text{ planar }G:$
$$\duality[15mm]{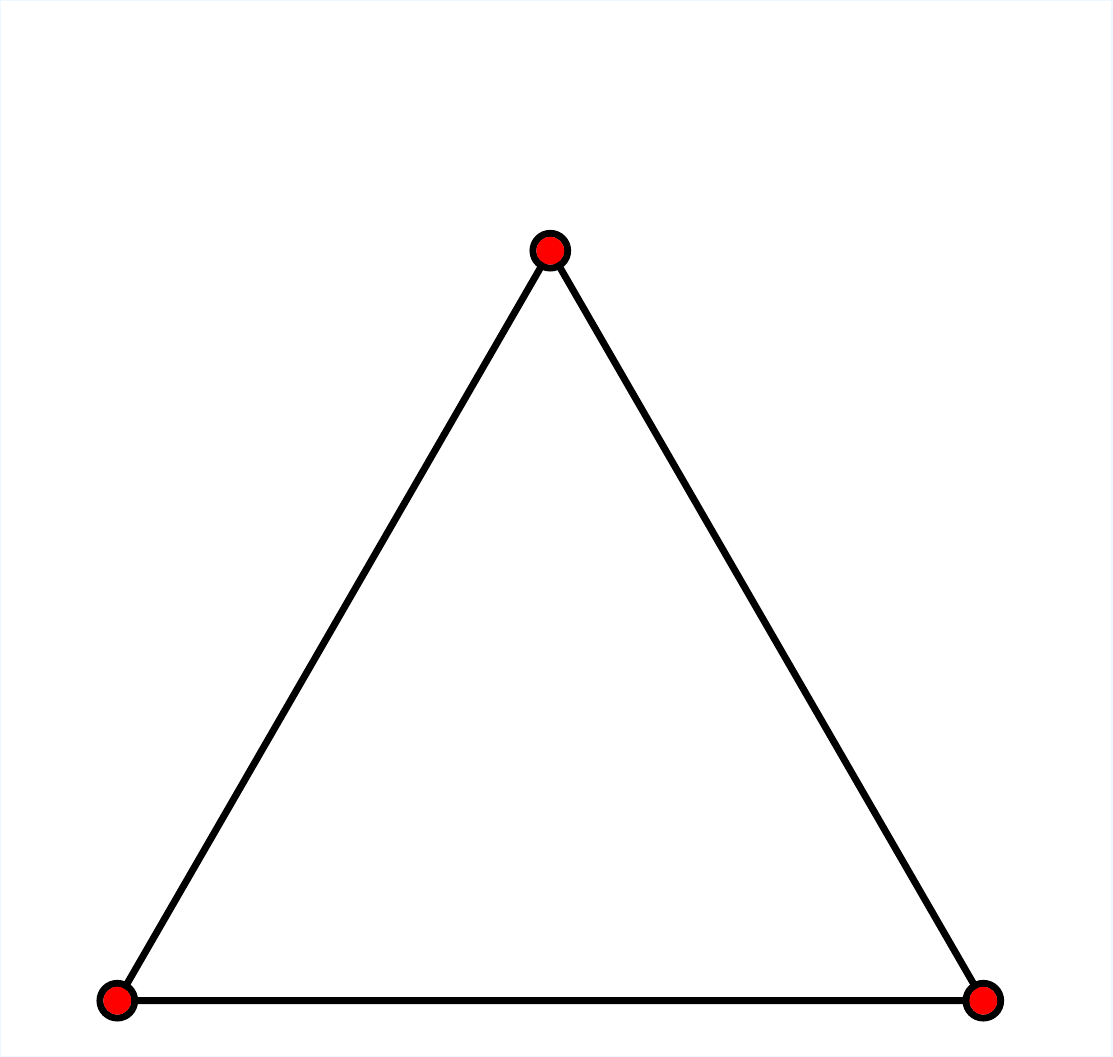}{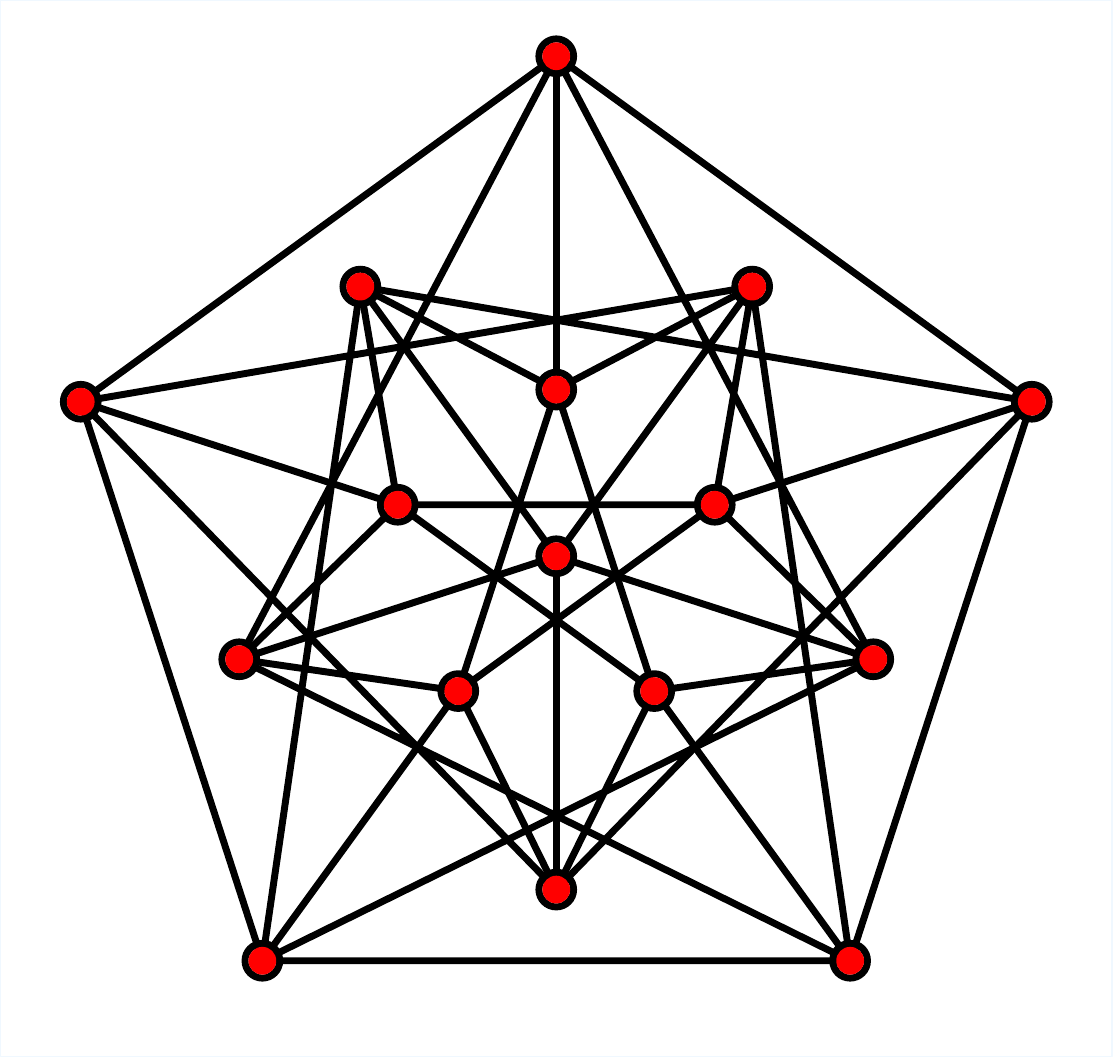}.$$
Note that this restricted homomorphism duality extends to the class of all graphs excluding  $K_5$ as a minor \cite{Naserasr20095789}.
\end{example}

\begin{example}
A restricted homomorphism duality for toroidal graphs follows from the existence of a finite set of obstructions for $5$-coloring
 proved by Thomassen in \cite{Thomassen199411}: Noticing that all the obstructions shown Fig.~\ref{fig:6crit}   are homomorphic images of one of them, namely $C_1^3$.

\begin{figure}[h!]
\begin{center}
\includegraphics[width=.75\textwidth]{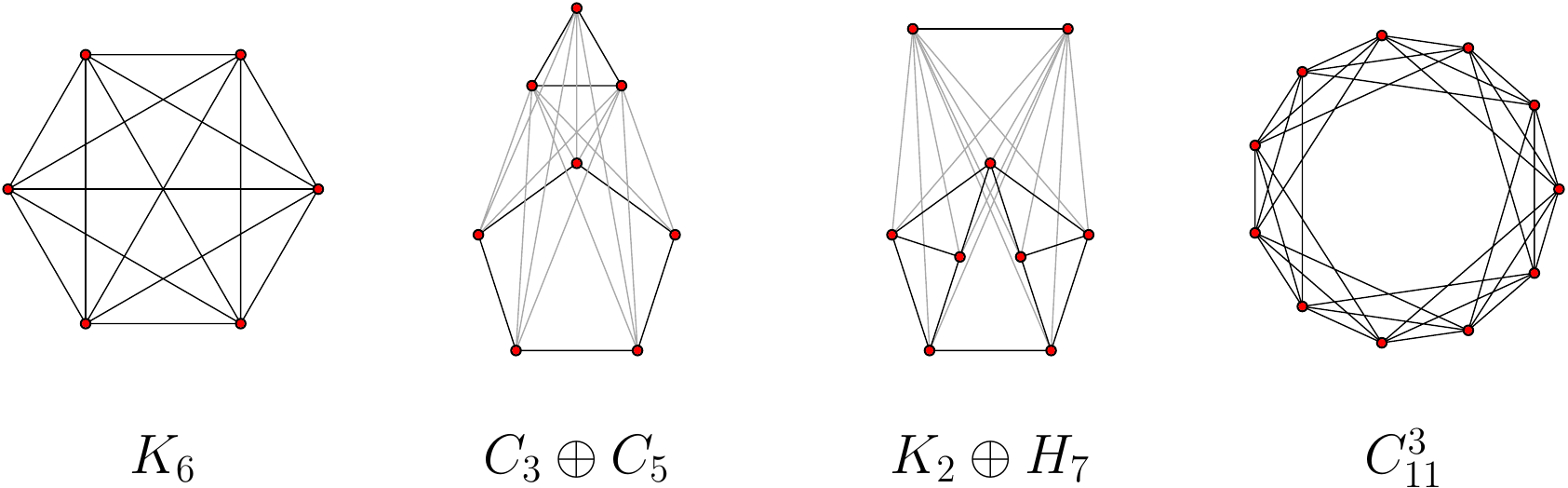}
\end{center}
\caption{The $6$-critical graphs for the torus.}
\label{fig:6crit}
\end{figure}

Thus we get the following restricted homomorphism duality.

$\forall\text{ toroidal }G:$
$$\duality[15mm]{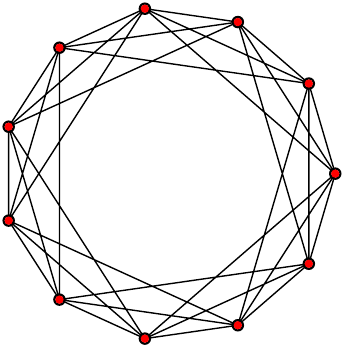}{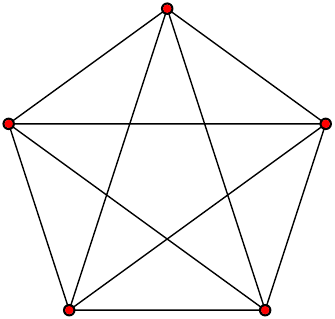}$$
\end{example}

\begin{definition}
 A class $\mathcal C$ with the property that every connected graph $F$ has a dual for $\mathcal C$ is said to have {\em all restricted dualities}.
 \end{definition}

In \cite{Taxi_tdepth} we proved, using low tree-depth decomposition, that for  every proper minor closed class $\mathcal C$ has all restricted dualities. We generalized in \cite{POMNIII} this result to classes with bounded expansions. We briefly outline this.

In the study of restricted homomorphism dualities, a main tool appeared to be  notion of $t$-approximation:

\begin{definition}
Let $G$ be a graph and let $t$ be a positive integer. A graph $H$ is
a {\em $t$-approximation} of $G$ if $G$ is  homomorphic to
$H$ (i.e. $G\rightarrow H$) and every subgraph of $H$ of order at most $t$ is homomorphic to $G$.
\end{definition}

Indeed the following theorem is proved in \cite{FO_CSP}:
\begin{theorem}
\label{thm:dual_approx}
Let $\mathcal C$ be a class of graphs. Then the following are equivalent:
\begin{enumerate}
\item The class $\mathcal C$ is bounded and has all restricted dualities (i.e. every connected graph $F$ has a dual for $\mathcal C$);
\item For every integer $t$ there is a constant $N(t)$ such that every graph $G\in\mathcal{C}$ has a $t$-approximation of order at most $N(t)$.
\end{enumerate}
\end{theorem}

The following lemma stresses the connection existing between $t$-approximation and low tree-depth decomposition:

\begin{theorem}[\cite{FO_CSP}]
\label{thm:chi2approx}
For every integer $t$ there exists a constant $C_t$  such that every
graph $G$ has a $t$-approximation $H$ with order
$$
|H|\leq C_t^{\chi_t(G)^t}.
$$
\end{theorem}
Hence we have the following corollary of Theorems~\ref{thm:dual_approx}, \ref{thm:chi2approx}, and~\ref{thm:chiBE}, which was originally proved in \cite{POMNIII}:
\begin{corollary}
Every class with bounded expansion has all restricted dualities.
\end{corollary}
The connection between classes with bounded expansion and restricted dualities appears to be even stronger, as witnessed by the following (partial) characterization theorem.
\begin{theorem}[\cite{FO_CSP}]
Let $\mathcal{C}$ be a topologically closed class of graphs (that is a class closed by the operation of graph subdivision).
Then the following are equivalent:
\begin{enumerate}
\item the class $\mathcal{C}$ has all restricted dualities;
\item the class $\mathcal{C}$ has bounded expansion.
\end{enumerate}
\end{theorem}
This theorem has also a variant in the context of directed graphs:
\begin{theorem}[\cite{FO_CSP}]
Let $\mathcal{C}$ be a class of directed graphs closed by reorientation.
Then the following are equivalent:
\begin{enumerate}
\item the class $\mathcal{C}$ has all restricted dualities;
\item the class $\mathcal{C}$ has bounded expansion.
\end{enumerate}
\end{theorem}
\section{Intermezzo: Low Tree-Depth Decomposition and Odd-Distance Coloring}
Let $n$ be an odd integer and let $G$ be a graph. The problem of
finding a coloring of the vertices of $G$ with minimum number of colors such that two vertices at distance $n$ are colored differently, called {$D_n$-coloring} of $G$,  was introduced in 1977 in Graph Theory Newsletter by E. Sampathkumar \cite{Sampathkumar1977} (see also \cite{jensen2011graph}). In \cite{Sampathkumar1977},
Sampathkumar claimed that every planar graph has a $D_n$-coloring for every odd integer $n$ with $5$ colors, and conjectured that $4$ colors suffice.
Unfortunately, the claimed result was flawed, as witnessed by the graph depicted on Figure~\ref{fig:D3col}, which needs $6$ colors for a $D_3$-coloring \cite{Sparsity}.

\begin{figure}[ht]
\begin{center}
\includegraphics[width=\textwidth]{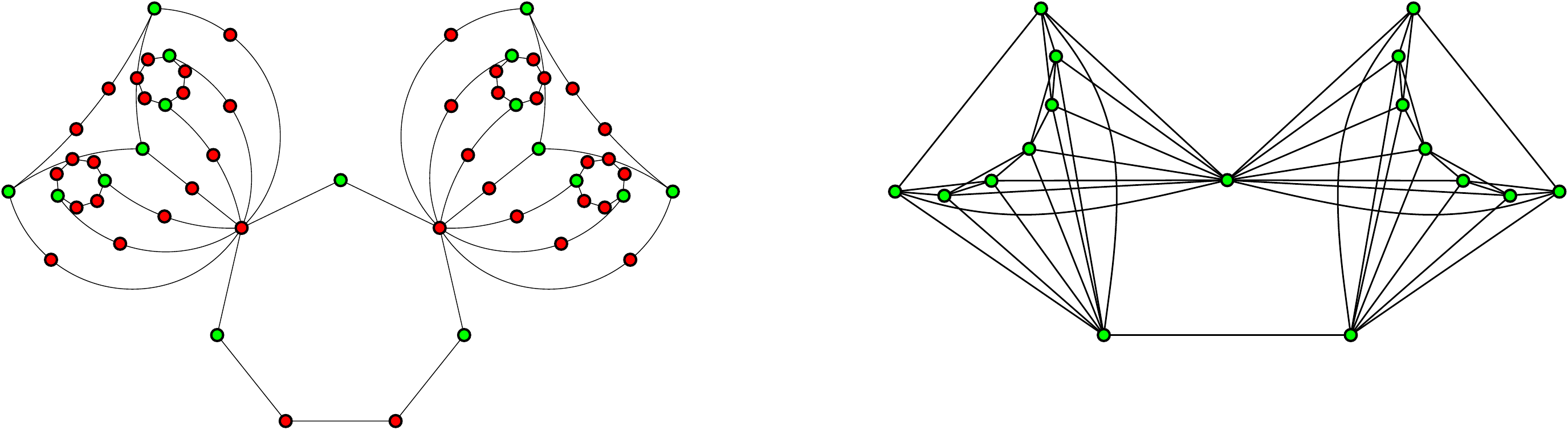}
\end{center}
\caption{On the left, a planar graph $G$ needing $6$-colors for a $D_3$-coloring. On the right, a witness: this a graph with vertex set $A\subset V(G)$ in which adjacent vertices are at distance $3$ in $G$, thus should get distinct colors in a $D_3$-coloring of $G$.}
\label{fig:D3col}
\end{figure}

Low tree-depth decomposition allows to prove that for any odd integer $n$, a fixed number of colors is sufficient for $D_n$-coloring planar graphs, and this results extends to all classes with bounded expansion.
\begin{theorem}[\cite{Sparsity}]
\label{thm:oddD}
For every class with bounded expansion $\mathcal C$ and every odd integer $n$ there exists a constant $N$ such that every graph $G\in\mathcal{C}$ has a $D_n$-coloring with at most $N$ colors.
\end{theorem}
The proof of Theorem~\ref{thm:oddD} relies on low tree-depth decomposition, and
 the bound $N$ given in \cite{Sparsity} for the number of colors sufficient for a $D_n$-coloring of a graph $G$ is double exponential in $\chi_n(G)$. Hence it is still not clear whether a uniform bound could exist for $D_n$-coloring of planar graphs.

\begin{problem}[van den Heuvel and Naserasr]
 Does there exist a constant $C$ such that for every odd integer $n$, it holds that every planar graph has a $D_n$-coloring with at most $C$ colors?
\end{problem}

Note that, however, there exists no bound for the {\em odd-distance coloring} of planar graphs, which requires that two vertices at odd distance get different colors. Indeed, one can construct outerplanar graphs having an arbitrarily large subset of vertices pairwise at odd distance (see Fig.~\ref{fig:oddcl}).

\begin{figure}[ht]
\begin{center}
\includegraphics[width=.35\textwidth]{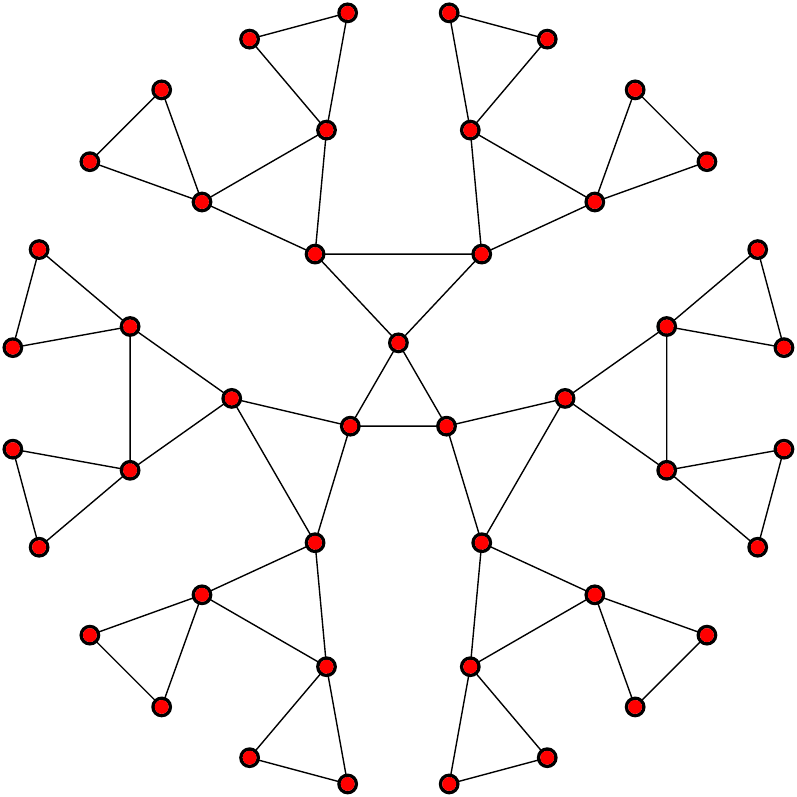}
\caption{There exist outerplanar graphs with arbitrarily large subset of vertices pairwise at odd distance. (In the figure, the vertices in the periphery are pairwise at distance $1$, $3$, $5$, or $7$.)}
\label{fig:oddcl}
\end{center}
\end{figure}

However, no construction requiring a large number of colors without having a large set of vertices pairwise at odd-distance is known. Hence the following problem.

\begin{problem}[Thomass\'e]
Does there exist a function $f:\bbbn\rightarrow\bbbn$ such that every planar graph without $k$ vertices pairwise at odd distance has an odd-distance coloring with at most $f(k)$ colors?
\end{problem}

\section{Low Tree-Depth Decomposition and Density of Shallow Minors, Shallow Topological Minors, and Shallow Immersions}
\label{sec:taxonomy}

Classes with bounded expansion, which have been introduced in \cite{POMNI}, may be viewed as a relaxation of the notion of proper minor closed class. The original definition of classes with bounded expansion relates to the notion of shallow minor, as introduced by Plotkin, Rao, and Smith
\cite{shallow}.

\begin{definition}
Let $G,H$ be graphs with $V(H)=\{v_1,\dots,v_h\}$ and let $r$ be an integer.
A graph $H$ is a {\em shallow minor} of a graph $G$ {\em at depth} $r$, if
there exists disjoint subsets $A_1,\dots,A_h$ of $V(G)$ such that
(see Fig.~\ref{fig:shm})
\begin{itemize}
  \item the subgraph of $G$ induced by $A_i$ is connected and as radius at most
  $r$,
  \item if $v_i$ is adjacent to $v_j$ in $H$, then some vertex in $A_i$ is
  adjacent in $G$ to some vertex in $A_j$.
\end{itemize}
\end{definition}

\begin{figure}[ht]
\begin{center}
\includegraphics[width=.75\textwidth]{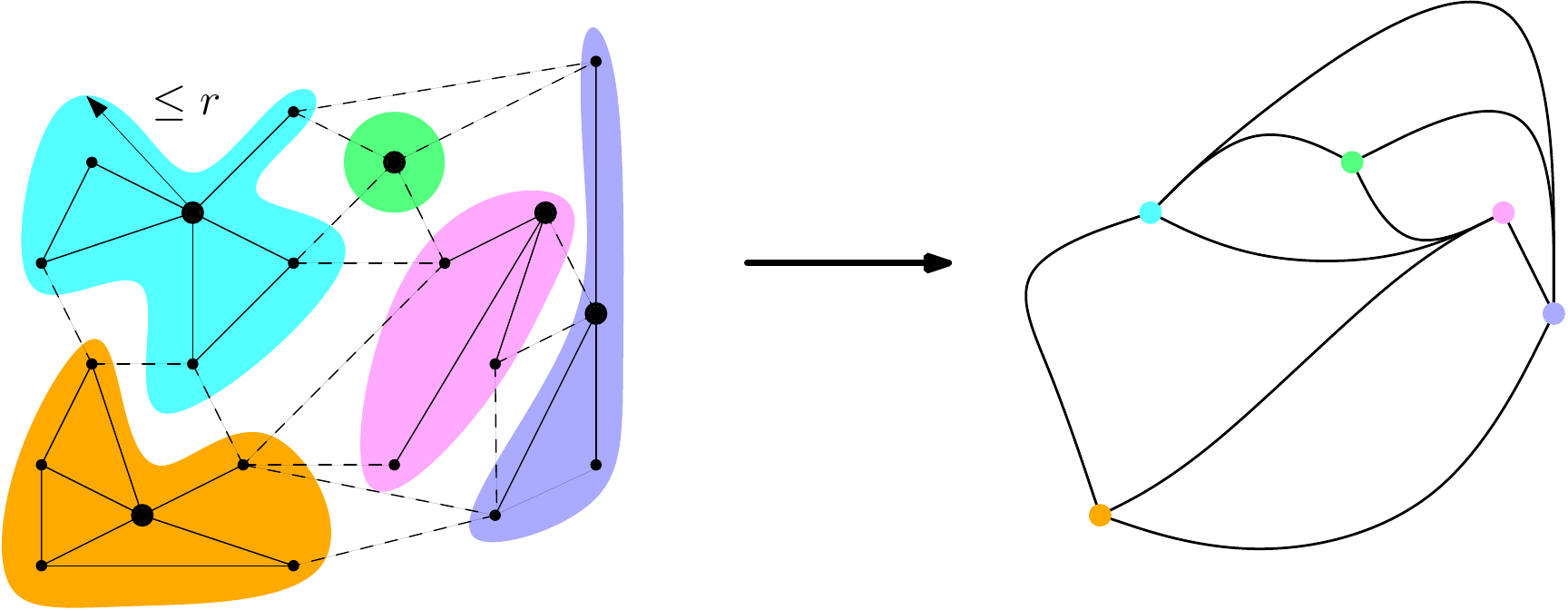}
\caption{A shallow minor}
\label{fig:shm}
\end{center}
\end{figure}

We denote \cite{POMNI, Sparsity} by $G\shm r$  the class of the (simple) graphs
which are shallow minors of $G$ at depth $r$,
and we denote by $\rdens{r}(G)$ the maximum density of a graph in
$G\shm r$, that is:
$$\rdens{r}(G)=\max_{H\in G\shm r}\frac{\|H\|}{|H|}$$
A class $\mathcal C$ has {\em bounded expansion} if
$\sup_{G\in\mathcal{C}}\rdens{r}(G)<\infty$  for each value of $r$.

Considering shallow minors may, at first glance, look arbitrary. Indeed
one can define as well the notions of shallow topological minors and shallow immersions:

\begin{definition}
A graph $H$ is a {\em shallow topological minor} at depth $r$ of a graph $G$ if some subgraph of $G$ is isomorphic to a subdivision of $H$ in which every edge has been subdivided at most $2r$ times (see Fig.~\ref{fig:stm}).

\begin{figure}[ht]
\begin{center}
\includegraphics[width=.75\textwidth]{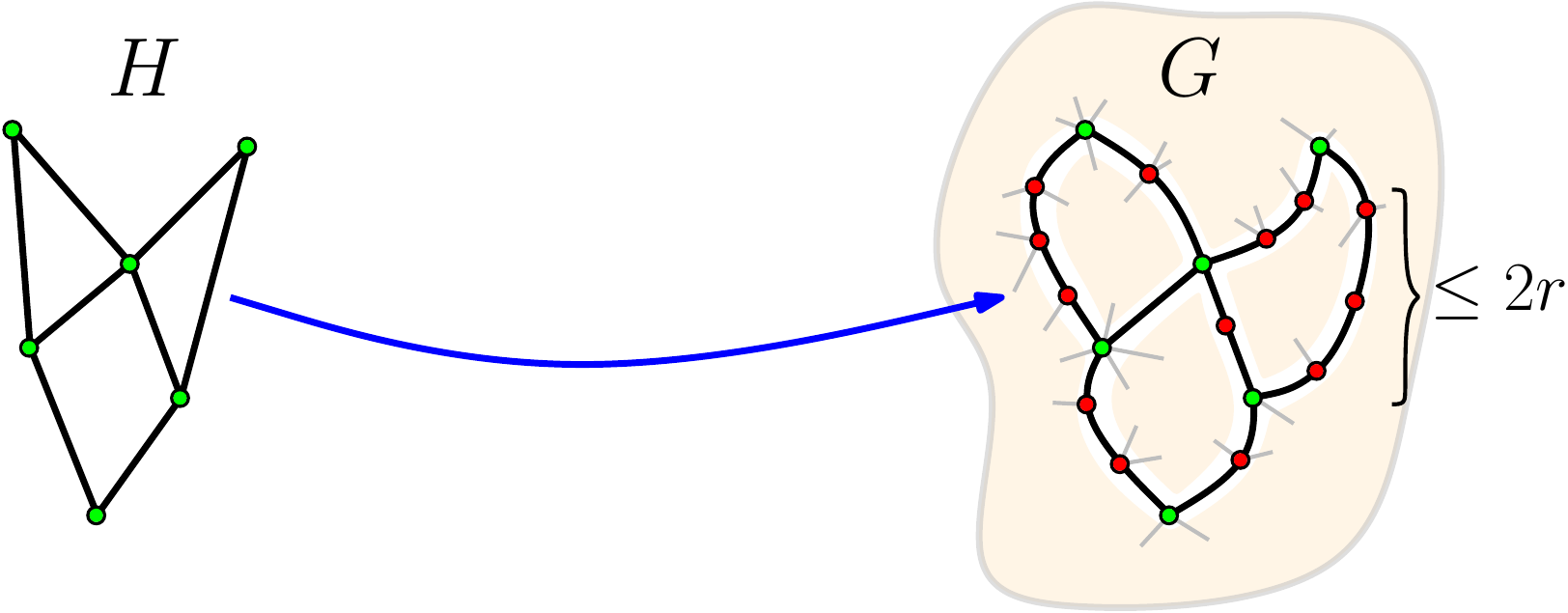}
\caption{$H$ is a shallow topological minor of $G$ at depth $r$}
\label{fig:stm}
\end{center}
\end{figure}

We denote \cite{POMNI, Sparsity} by $G\shtm r$  the class of the (simple) graphs
which are shallow topological minors of $G$ at depth $r$,
and we denote by $\trdens{r}(G)$ the maximum density of a graph in
$G\shtm r$, that is:
$$\trdens{r}(G)=\max_{H\in G\shtm r}\frac{\|H\|}{|H|}$$
\end{definition}

Note that shallow topological minors can be alternatively defined by considering how a graph $H$ can be topologically embedded in a graph $G$:
a graph $H$ with vertex set $V(H)=\{a_1,\dots,a_k\}$ is a shallow topological minor of a graph $G$ at depth $r$ is there exists vertices $v_1,\dots,v_k$ in $G$ and a family $\mathcal P$ of paths of $G$ such that
\begin{itemize}
\item two vertices $a_i$ and $a_j$ are adjacent in $H$ if and only if there is a path in $\mathcal{P}$ linking $v_i$ and $v_j$;
\item no vertex $v_i$ is interior to a path in $\mathcal{P}$;
\item the paths in $\mathcal{P}$ are internally vertex disjoint;
\item every path in $\mathcal{P}$ has length at most $2r+1$.
\end{itemize}

We can similarly define the notion of shallow immersion:
\begin{definition}
A graph $H$ with vertex set $V(H)=\{a_1,\dots,a_k\}$ is a {\em shallow immersion} of a graph $G$ at depth $r$ is there exists vertices $v_1,\dots,v_k$ in $G$ and a family $\mathcal P$ of paths of $G$ such that
\begin{itemize}
\item two vertices $a_i$ and $a_j$ are adjacent in $H$ if and only if there is a path in $\mathcal{P}$ linking $v_i$ and $v_j$;
\item the paths in $\mathcal{P}$ are edge disjoint;
\item every path in $\mathcal{P}$ has length at most $2r+1$;
\item no vertex of $G$ is internal to more than $r$ paths in $\mathcal{P}$.
\end{itemize}
We denote \cite{POMNI, Sparsity} by $G\shim r$  the class of the (simple) graphs
which are shallow immersions of $G$ at depth $r$,
and we denote by $\irdens{r}(G)$ the maximum density of a graph in
$G\shim r$, that is:
$$\irdens{r}(G)=\max_{H\in G\shim r}\frac{\|H\|}{|H|}$$
\end{definition}

It appears that although minors, topological minors, and immersions behave very differently, their shallow versions are deeply related, as witnessed by the following theorem:
\begin{theorem}[\cite{Sparsity}]
\label{thm:BE}
Let $\mathcal{C}$ be a class of graphs. Then the following are equivalent:
\begin{enumerate}
\item the class $\mathcal{C}$ has bounded expansion;
\item for every integer $r$ it holds $\sup_{G\in\mathcal{C}}\rdens{r}(G)<\infty$;
\item for every integer $r$ it holds $\sup_{G\in\mathcal{C}}\trdens{r}(G)<\infty$;
\item for every integer $r$ it holds $\sup_{G\in\mathcal{C}}\irdens{r}(G)<\infty$;
\item for every integer $r$ it holds $\sup_{H\in\mathcal{C}\shm r}\chi(H)<\infty$;
\item for every integer $r$ it holds $\sup_{H\in\mathcal{C}\shtm r}\chi(H)<\infty$;
\item for every integer $r$ it holds $\sup_{H\in\mathcal{C}\shim r}\chi(H)<\infty$.
\end{enumerate}
\end{theorem}

In the above theorem, we see that not only shallow minors, shallow topological minors, and shallow immersions behave closely, but that the (sparse) graph density $\|G\|/|G|$ and the chromatic number $\chi(G)$ of a graph $G$ are also related. This last relation is intimately
related to the following result of Dvor\'ak \cite{Dvo2007}.
\begin{lemma}
\label{lem:degchr}
Let $c\geq 4$ be an integer and let $G$ be a graph with average degree
$d>56(c-1)^2\frac{\log (c-1)}{\log c-\log (c-1)}$.
Then the graph $G$ contains a subgraph $G'$ that is the $1$-subdivision of a graph with chromatic number $c$.
\end{lemma}

It follows from Theorem~\ref{thm:BE} that the  notion of class with bounded expansion is quite robust. Not only classes with bounded expansion can be defined by edge densities and chromatic number,  but also by virtually all common combinatorial parameters \cite{Sparsity}.

If one considers the clique number instead of the density or the chromatic number, then a different type of classes is defined:
\begin{definition}
A class of graph $\mathcal{C}$ is {\em somewhere dense} if
there exists an integer $p$ such that every clique is a shallow topological minor at depth $p$ of some graph in $\mathcal C$ (in other words, $\mathcal{C}\shtm p$ contain all graphs); the class
$\mathcal{C}$ is {\em nowhere dense} if it is not somewhere dense.
\end{definition}

Similarly that Theorem~\ref{thm:BE}, we have several characterizations of nowhere dense classes.

\begin{theorem}[\cite{Sparsity}]
\label{thm:ND}
Let $\mathcal{C}$ be a class of graphs. Then the following are equivalent:
\begin{enumerate}
\item the class $\mathcal{C}$ is nowhere dense;
\item for every integer $r$ it holds $\limsup_{G\in\mathcal{C}}
\frac{\log\rdens{r}(G)}{\log |G|}=0$;
\item for every integer $r$ it holds $\limsup_{G\in\mathcal{C}}
\frac{\log\trdens{r}(G)}{\log |G|}=0$ ;
\item for every integer $r$ it holds $\limsup_{G\in\mathcal{C}}
\frac{\log\irdens{r}(G)}{\log |G|}=0$;
\item for every integer $r$ it holds $\sup_{H\in\mathcal{C}\shm r}\omega(H)<\infty$;
\item for every integer $r$ it holds $\sup_{H\in\mathcal{C}\shtm r}\omega(H)<\infty$;
\item for every integer $r$ it holds $\sup_{H\in\mathcal{C}\shim r}\omega(H)<\infty$.
\end{enumerate}
\end{theorem}

Note that every class with bounded expansion is nowhere dense.
As mentioned in Theorem~\ref{thm:chiBE}, classes with bounded expansion are also characterized by the fact that they allow low tree-depth decompositions with bounded number of colors. A similar statement holds for nowhere dense classes:

 Precisely, we have the following:
 \begin{theorem}
\label{thm:chiND}
Let $\mathcal{C}$ be a class of graphs, then the following are equivalent:
\begin{enumerate}
\item for every integer $p$ it holds $\limsup_{G\in\mathcal{C}}\frac{\chi_p(G)}{\log |G|}=0$;
\item the class $\mathcal{C}$ is nowhere dense.
\end{enumerate}
\end{theorem}

The direction bounding $\chi_p(G)$ of both Theorem~\ref{thm:chiBE} and~\ref{thm:chiND} follow from the next more precise result:

\begin{theorem}[\cite{Sparsity}]
\label{thm:chibound}
For every integer $p$ there is a polynomial $P_p$ (${\rm deg}\,P_p\approx 2^{2^p}$) such that for
every graph $G$ it holds
$$
\chi_p(G)\leq P_p(\trdens{2^{p-2}+1}(G)).
$$
\end{theorem}
Note that the original proof given in \cite{POMNI} gave a slightly weaker bound, and that an alternative proof of this result has been obtained by Zhu \cite{Zhu2008}, in a paper relating low tree-depth decomposition with the generalized coloring numbers introduced by Kierstead and Yang \cite{Kierstead2003}.

\section{Low Tree-Depth Decomposition and Covering}
In a low treedepth decomposition of a graph $G$ by $N$ colors and for parameter $t$,
the subsets of $t$ colors define a disjoint union of clusters that cover the graph, such that each cluster has tree-depth at most $t$, every vertex belongs to at most $\binom{N}{t}$ clusters, and every connected subgraph of order $t$ is included in at least one cluster.

It is natural to ask whether the condition that such a covering comes from a coloring could be dropped.

\begin{theorem}
Let $\mathcal C$ be a monotone class.

Then $\mathcal C$ has bounded expansion if and only if  there exists a  function $f$ such that for every integer $t$, every graph $G\in\mathcal C$ has a
covering $C_1,\dots, C_k$ of its vertex set such that
\begin{itemize}
\item each $C_i$ induces a connected subgraph with tree-depth at most $t$;
\item every vertex belongs to at most $f(t)$ clusters;
\item every connected subgraph of order at most $t$ is included in at least one cluster.
\end{itemize}
\end{theorem}
\begin{proof}
One direction is a direct consequence of Theorem~\ref{thm:chiBE}.
Conversely, assume that the class $\mathcal C$ does not have bounded expansion. Then there exists an integer $p$ such that
for every integer $d$ the class $\mathcal C$ contains the $p$-th subdivision of a graph $H_d$ with average degree at least $d$. Moreover, it is a standard argument that we can require $H_d$ to be
bipartite (as every graph with average degree $2d$ contains a bipartite subgraph with average degree at least $d$).

Let $t=2(p+1)$ and let $d=2f(t)+1$. Assume for contradiction that there exist clusters $C_1,\dots,C_k$ as required, then we can cover $H_d$ by clusters $C_1',\dots,C_k'$ such that
each $C_i'$ induces a star (possibly reduced to an edge),
every vertex belongs to at most $f(t)$ clusters, and
every edge is included in at least one cluster.
If an edge $\{u,v\}$ of $H_d$ is included in more than two clusters, it is easily checked that (at least) one of $u$ and $v$ can be safely removed from one of the cluster. Hence we can assume that each edge of $H_d$ is covered exactly once. To each cluster $C_i'$ associates the center of the star induced by $C_i'$ (or an arbitrary vertex of $C_i'$ if $C_i'$ has cardinality $2$) and orient the edges of the star induced by $C_i'$ away from the center. This way, every edge is oriented once and every vertex gets indegree at most $f(t)$. However, summing the indegrees we get $f(t)\geq d/2$, a contradiction.
\end{proof}
It is natural to ask whether similar statements would hold, if we weaken the condition that each cluster has tree-depth at most $t$ while we strengthen the condition that every connected subgraph of order at most $t$ is included in some cluster.
Namely, we consider the question whether a similar statement holds
if we allow each cluster to have radius at most $2t$ while requiring that every $t$-neighborhood is included in some cluster.
In the context of their solution of model checking problem for nowhere dense classes, Grohe, Kreutzer and Siebertz
introduced in \cite{Grohe2013} the notion of $r$-neighborhood cover and proved that
nowhere dense classes admit such cover with small maximum degree, and proved that nowhere dense classes and bounded expansion classes admit such nice covering.

Precisely,
for $r\in\bbbn$, an {\em $r$-neighborhood cover} $\mathcal{X}$ of a graph $G$ is a set of connected subgraphs of $G$ called {\em clusters}, such that for every vertex $v\in V(G)$ there is some
$X\in\mathcal{X}$ with $N_r(v)\subseteq X$.
The {\em radius} ${\rm rad}(\mathcal{X})$ of a cover $\mathcal{X}$ is the maximum radius of its clusters. The {\em degree} $d^\mathcal{X}(v)$ of $v$ in $\mathcal{X}$ is the number of clusters that contain $v$. The {\em maximum degree} $\Delta(\mathcal{X})=\max_{v\in V(G)}d^\mathcal{X}(v)$.
For a graph $G$ and $r\in\bbbn$ we define
$\tau_r(G)$ as the minimum maximum degree of an $r$-neighborhood cover of radius at most $2r$ of $G$.

The following theorem is proved in \cite{Grohe2013}.
\begin{theorem}
\label{thm:1b}
Let $\mathcal{C}$ be a class of graphs with bounded expansion. Then there is a function $f$ such that for all $r\in\bbbn$ and all graphs $G\in\mathcal{C}$ , it holds
$\tau_r(G)\leq f(r)$.
\end{theorem}

In order to prove the converse statement, we shall need the following result of K\"uhn and Osthus \cite{K`uhn2004}:
\begin{theorem}
\label{thm:ko}
For every $k$ there exists $d = d(k)$ such that every graph of average degree at least $d$ contains a subgraph of average degree at least $k$ whose girth is at least six.
\end{theorem}

We are now ready to turn Theorem~\ref{thm:1b} into a characterization theorem of classes with bounded expansion.

\begin{theorem}
Let $\mathcal{C}$  be an infinite monotone class of graphs. Then
$\mathcal C$ has bounded expansion if and only if, for every integer $r$ it holds
$$\sup_{G\in\mathcal{C}} \tau_r(G)<\infty.$$
\end{theorem}
\begin{proof}
One direction follows from Theorem~\ref{thm:1b}. For the other direction, assume that the class $\mathcal C$ does not have bounded expansion. Then there exists an integer $p$ such that
for every integer $n$, $\mathcal C$ contains the $p$-th subdivision
 of a graph $G_n$  with average degree  at least $n$.

Let $d\in\bbbn$. According to Theorem~\ref{thm:ko}, there exists $N(d)$ such that every graph with average degree at least $N(d)$ contains a subgraph of girth $6$ and average degree at least $d$.
We deduce that $\mathcal C$ contains the $p$-th subdivision
$H_d'$ of a graph $H_d$  with girth at least $6$ and average degree at least $d$. As in the proof of Theorem~\ref{thm:2}, we get
$$\sup_{G\in\mathcal{C}} \tau_{p+1}(G)
\geq \sup_{d} \tau_{p+1}(H_d')
\geq \sup_{d}\tau_{1}(H_d)
\geq \sup_{d}\frac{\|H_d\|}{|H_d|}=\infty.$$
\end{proof}

Also, similar statements exist for nowhere dense classes:

\begin{theorem}
A hereditary class $\mathcal C$ is nowhere dense if there exists a
function $f$ such that for every integer $t$ and every $\epsilon>0$, every graph $G\in\mathcal C$ of order $n\geq f(t,\epsilon)$ has a
covering $C_1,\dots, C_k$ of its vertex set such that
\begin{itemize}
\item each $C_i$ induces a connected subgraph with tree-depth at most $t$;
\item every vertex belongs to at most $n^\epsilon$ clusters;
\item every connected subgraph of order at most $t$ is included in at least one cluster.
\end{itemize}
\end{theorem}
\begin{proof}
One direction directly follows from Theorem~\ref{thm:chiND}. For the reverse direction, assume that $\mathcal C$ is not nowhere dense. Then there exists $p$ such that for every $n\in\bbbn$, the class $\mathcal C$ contains a graph
$G_n$ having the $p$-th subdivision of $K_n$ as the spanning subgraph.
Assume that a covering exists for $t=3p+3$. Then every $p$-subdivided triangle of $K_n$ is included in some cluster.
As the $p$-subdivided $K_n$ includes $\binom{n}{3}$ triangles, and
as there are at most $n^{1+\epsilon}$ clusters including some principal vertex of the subdivided $K_n$ (which is necessary to include some subdivided triangle), some cluster $C$ includes at least
$n^{2-\epsilon}$ triangles. It follows that the subgraph induced by $C$ has a minor $H$ of order at most $n$ with at least $n^{2-\epsilon}$ triangles. However, as tree-depth is minor monotone, the graph $H$ has tree-depth at most $t$ hence is $t$-degenerate thus cannot contain more than $\binom{t}{2}n$ triangles. Whence we are led to a contradiction if $n>\binom{t}{2}^{\frac{1}{1-\epsilon}}$.
\end{proof}

\begin{theorem}[ \cite{Grohe2013}]
\label{thm:1}
Let $\mathcal{C}$ be a nowhere dense class of graphs. Then there is a function $f$ such that for all $r\in\bbbn$ and $\epsilon>0$ and all graphs $G\in\mathcal{C}$ with $n\geq f(r,\epsilon)$ vertices, it holds
$\tau_r(G)\leq n^\epsilon$.
\end{theorem}

In other words, every infinite nowhere dense class of graphs $\mathcal C$ is such that
$$\adjustlimits\sup_{r\in\bbbn}\limsup_{G\in\mathcal{C}}\frac{\log \tau_r(G)}{\log |G|}=0.$$

We shall deduce from this theorem the following characterization of nowhere dense classes of graphs.

\begin{theorem}
\label{thm:2}
Let $\mathcal{C}$  be an infinite monotone class of graphs. Then
$$\adjustlimits\sup_{r\in\bbbn}\limsup_{G\in\mathcal{C}}\frac{\log \tau_r(G)}{\log |G|}$$
is either $0$ if $\mathcal C$ is nowhere dense, at at least $1/3$ if $\mathcal{C}$ is somewhere dense.
\end{theorem}

This theorem will directly follow from Theorem~\ref{thm:1} and the following two lemmas.

\begin{lemma}
\label{lem:1}
Let $G$ be a graph of girth at least $5$. Then it holds
$$
\tau_1(G)\geq \nabla_0(G),
$$
where
$$\nabla_0(G)=\max_{H\subseteq G}\frac{\|H\|}{|H|}.$$
\end{lemma}
\begin{proof}
Let $\mathcal{X}$ be a $1$-neighborhood cover of radius at most $2$ of $G$ with maximum degree $\tau_1(G)$. Let $X_1,\dots,X_k$ be the clusters of $\mathcal{X}$. For an edge $e=\{u,v\}$, let $i\leq k$ be the minimum integer such that $N_1(u)$ or $N_1(v)$ is included in $X_i$. Let $c_i$ be a center of $X_i$. Then $e$ belongs
to a path of length at most $2$ with endpoint $c_i$. We orient $e$ according to the orientation of this path away from $c_i$. Note that by the process, we orient every edge, and that every vertex $v$ gets at most one incoming edge by cluster that contains $v$. Hence
we constructed an orientation of $G$ with maximum degree at most $\tau_1(G)$. As the maximum indegree of an orientation of $G$ is at least $\nabla_0(G)$, we get $\tau_1(G)\geq \nabla_0(G)$.
\end{proof}

We deduce the following
\begin{lemma}
\label{lem:2}
Let $\mathcal{C}$ be a monotone somewhere dense class of graphs. Then
$$\adjustlimits\sup_{r\in\bbbn}\limsup_{G\in\mathcal{C}}\frac{\log \tau_r(G)}{\log |G|}\geq \frac{1}{3}.$$
\end{lemma}
\begin{proof}
A $\mathcal{C}$ is monotone and somewhere dense, there exists integer $p\geq 0$ such that for every $n\in\bbbn$, the $p$-th subdivision ${\rm Sub}_p(K_n)$ of $K_n$ belongs to $\mathcal C$.
For $n\in\bbbn$, let $H_n$ be a graph of girth at least $5$, with order $|H_n|\sim n$ and size $\|H_n\|\sim n^{3/2}$.
If $p=0$, then according to Lemma~\ref{lem:1} it holds
\begin{align*}
\adjustlimits\sup_{r\in\bbbn}\limsup_{G\in\mathcal{C}}\frac{\log \tau_r(G)}{\log |G|}&\geq \limsup_{G\in\mathcal{C}}\frac{\log \tau_1(G)}{\log |G|}\\
&\geq \lim_{n\rightarrow\infty}\frac{\log \nabla_0(H_n)}{\log |H_n|}\\
&\geq \lim_{n\rightarrow\infty}\frac{\log \|H_n\|-\log |H_n|}{\log |H_n|}
= \frac{1}{2}.
\end{align*}
Thus assume $p\geq 1$. Denote by $H_n'$ the $p$-th subdivision
of $H_n$, where we identify $V(H_n)$ with a subset of $V(H_n')$ for convenience. Then $|H_n|\sim pn^{3/2}$.
Let $\mathcal X=\{X_1,\dots,X_k\}$ be a $(p+1)$-neighborhood cover of radius at most $2(p+1)$ of $H_n'$ with maximum degree $\tau_{p+1}(H_n')$.
Let $c_i$ be a center of cluster $X_i$, and let $d_i$ be a vertex of $H_n$ at minimal distance of $c_i$ in $H_n'$. It is easily checked that
there exists a cluster $X_i'$ with center $d_i$ and radius $2(p+1)$ such that $X_i\cap V(H_n)=X_i'\cap V(H_n)$. Define $Y_i=X_i'\cap V(H_n)$. As $\mathcal X$ is a $(p+1)$-neighborhood cover of radius at most $2(p+1)$ of $H_n'$ with maximum degree $\tau_1(H_n')$,
the cover $\mathcal Y=\{Y_i\}$ is a $1$-neighborhood cover of radius $2$ of $H_n$  with maximum degree $\tau_{p+1}(H_n')$.
Hence $\tau_1(H_n)\leq \tau_{p+1}(H_n')$. Thus it holds
\begin{align*}
\adjustlimits\sup_{r\in\bbbn}\limsup_{G\in\mathcal{C}}\frac{\log \tau_r(G)}{\log |G|}
&\geq \lim_{n\rightarrow\infty}\frac{\log \tau_{p+1}(H_n')}{\log |H_n'|}\\
&\geq \lim_{n\rightarrow\infty}\frac{\log \tau_{1}(H_n)}{\log |H_n'|}\\
&\geq \lim_{n\rightarrow\infty}\frac{\log \|H_n\|-\log |H_n|}{\log |H_n'|}=\frac{1}{3}.
\end{align*}
\end{proof}
\section{Algorithmic Applications of Low Tree-Depth Decomposition}
\label{sec:algo}
Theorem~\ref{thm:chibound} has the following algorithmic version.
\begin{theorem}[\cite{Sparsity}]
\label{thm:chialgo}
There exist polynomials $P_p$ (${\rm deg}\,P_p\approx 2^{2^p}$) and an algorithm that computes, for input graph $G$ and integer $p$, a low tree-depth decomposition of $G$ with parameter $p$ using
$N_p(G)$ colors in time $O(N_p(G)\,|G|)$, where
$$
\chi_p(G)\leq N_p(G)\leq P_p(\trdens{2^{p-2}+1}(G)).
$$
\end{theorem}

It is not surprising that
low tree-depth decompositions have immediately found several algorithmic applications \cite{Taxi_stoc06, POMNII}.

As noticed in \cite{chrobak}, the existence of an orientation of planar graphs with bounded out-degree allows for a planar graph $G$ (once such an orientation has been computed for $G$) an easy $O(1)$ adjacency test, and an enumeration of all the triangles of $G$ in linear time.

 For a fixed pattern $H$,
the problem is to check whether an input graph $G$ has an induced subgraph
isomorphic to $H$ is called the {\em subgraph isomorphism problem}.
This problem is known to have complexity at most
$O(n^{\omega l/3})$ where $l$ is the order of $H$ and where $\omega$ is the
exponent of square matrix fast multiplication algorithm \cite{NP85} (hence
$O(n^{0.792\ l})$ using the fast matrix algorithm of \cite{coppersmith90}). The
particular case of subgraph isomorphism in planar graphs have been studied by
Plehn and Voigt \cite{plehn91}, Alon \cite{alon95} with super-linear bounds and
then by Eppstein \cite{Epp-SODA-95,Epp-JGAA-99} who gave the first linear
time algorithm for fixed pattern $H$ and $G$ planar. This was extended to
graphs with bounded genus in \cite{Epp-Algo-00}.
We further generalized this result to classes with bounded expansion \cite{POMNII}:
\begin{theorem}
\label{thm:count}
There is a function $f$ and an algorithm such that for every input graphs $G$ and $H$, counts the number of occurrences of $H$ is $G$ in time
$$O\bigl(f(H)\,(N_{|H|}(G))^{|H|}\,|G|\bigr),$$ where $N_p(G)$ is the number of colors computed by the algorithm in Theorem~\ref{thm:chialgo}.

In particular, for every fixed bounded expansion class (resp. nowhere dense class) $\mathcal C$ and every fixed pattern $H$, the number of occurrences of $H$ in a graph $G\in\mathcal C$ can be computed in linear time (resp. in time $O(|G|^{1+\epsilon})$ for any fixed $\epsilon>0$).
\end{theorem}

Theorem~\ref{thm:count} can be extended from the subgraph isomorphism problem to first-order model checking.

\begin{theorem}[\cite{DKT2}, see also \cite{Dawar2009}]
Let $\mathcal C$ be a class of graphs with bounded expansion, and
let $\phi$  be a first-order sentence (on the natural language of graphs). There exists a linear time  algorithm that decides whether a graph $G\in\mathcal C$ satisfies $\phi$.
\end{theorem}

The above theorem relies on low tree-depth decomposition. However, the next result, due to Kazana and Segoufin, is based on the notion of transitive fraternal augmentation, which was introduced in \cite{POMNI} to prove Theorem~\ref{thm:chibound}.

\begin{theorem}[\cite{Kazana2013}]
\label{thm:Kazana}
 Let $\mathcal C$ be a class of graphs with bounded expansion and let $\phi$ be a first-order formula. Then, for all $G\in\mathcal C$, we can compute the number $|\phi(G)|$ of satisfying assignements for $\phi$ in $G$ in  in time $O(|G|)$.

 Moreover, the set $\phi(G)$ can be enumerated in lexicographic order in constant time between consecutive outputs and linear time preprocessing time.
\end{theorem}

Eventually, the existence of efficient model checking algorithm has been extended to nowhere dense classes by
Grohe, Kreutzer, and Siebertz \cite{Grohe2013} using the notion
of $r$-neighborhood cover we already mentioned:

\begin{theorem}
\label{thm:FOND}
For every nowhere dense class $\mathcal C$ and every $\epsilon >0$, every property of graphs definable in first-order logic can be decided in time $O(n^{1+\epsilon})$ on $\mathcal{C}$.
\end{theorem}

However, it is still open whether a counting version of Theorem~\ref{thm:FOND} (in the spirit of Theorem~\ref{thm:Kazana}) holds.

\section{Low Tree-Depth Decomposition and Logarithmic Density of Patterns}

We have seen in the Section~\ref{sec:algo} that low tree-depth decomposition allows an easy counting of patterns. It appears that they also allow to prove some ``extremal'' results. A typical problem studied in extremal graph theory is to determine the maximum number of edges ${\rm ex}(n,H)$
a graph on $n$ vertices can contain without containing a subgraph isomorphic to $H$.
For non-bipartite graph $H$, the seminal result of Erd\H os and Stone \cite{erdos1946structure} gives a tight bound:

\begin{theorem}
$${\rm ex}(n,H)=\left(1-\frac{1}{\chi(H)-1}\right)\binom{n}{2}+o(n^2).$$
\end{theorem}

In the case of bipartite graphs, less is known. Let us mention the following result of
Alon, Krivelevich and Sudakov \cite{alon2003turan}

\begin{theorem}
Let $H$ be a bipartite graph with maximum degree $r$ on one side.
$${\rm ex}(n,H)= O(n^{2-\frac{1}{r}}).$$
\end{theorem}

The special case where $H$ is a subdivision of a complete graph will be of
prime interest in the study of nowhere dense classes. Precisely, denoting
${\rm ex}(n,K_t^{(\leq p)})$ the maximum number of edges a graph on $n$ vertices can contain without containing a subdivision of $K_t$ in which every edge is subdivided at most $p$ times,  Jiang \cite{Jiang} proved the following bound:

\begin{theorem}
For every integers $k,p$ it holds
$${\rm ex}(n,K_k^{(\leq p)})=O(n^{1+\frac{10}{p}}).$$
\end{theorem}

From this theorem follows that if a class $\mathcal{C}$ is such that
$\limsup_{G\in\mathcal{C}\shtm t}\frac{\log \|G\|}{\log |G|}>1+\epsilon$ then
$\mathcal C\shtm \frac{10t}{\epsilon}$ contains graphs with unbounded clique number. This property is a main ingredient in the proof of the following classification ``trichotomy'' theorem.

\begin{theorem}[\cite{ND_characterization}]
\label{thm:tri}
Let $\mathcal C$ be an infinite class of graphs. Then
\begin{equation*}
 \adjustlimits\sup_{t}\limsup_{G\in\mathcal C \shtm t}\frac{\log\|G\|}{\log|G|}\in\{-\infty,0,1,2\}.
\end{equation*}

Moreover, $\mathcal{C}$ is nowhere dense if and only if
$\adjustlimits\sup_{t}\limsup_{G\in\mathcal C \shtm t}\frac{\log\|G\|}{\log|G|}\leq 1$.
\end{theorem}

Note that the property that the logarithmic density of edges is integral
needs to consider all the classes $\mathcal C\shtm t$. For instance, the class
$\mathcal D$ of graphs
with no $C_4$ has a bounding logarithmic edge density of $3/2$, which jumps to $2$
when on considers $\mathcal{D}\shtm 1$.

Using low tree-depth decomposition, it is possible to extend
Theorem~\ref{thm:tri} to other pattern graphs:

\begin{theorem}[\cite{Taxi_hom}]
\label{thm:countF}
For every infinite class of graphs $\mathcal C$ and every graph $F$
$$
\adjustlimits \lim_{i\rightarrow\infty}\limsup_{G\in\mathcal C\shtm i}\frac{\log (\#F\subseteq G)}{\log |G|}\in\{-\infty,0,1,\dots,\alpha(F),|F|\},$$
where $\alpha(F)$ is the stability number of $F$.

Moreover, if $F$ has at least one edge, then $\mathcal{C}$ is nowhere dense if and only if $\adjustlimits \lim_{i\rightarrow\infty}\limsup_{G\in\mathcal C\shtm i}\frac{\log (\#F\subseteq G)}{\log |G|}\leq\alpha(F)$.
\end{theorem}

The main ingredient in the proof of this theorem is the analysis of local configurations, called $(k,F)$-sunflowers (see Fig.~\ref{fig:sunflower}). Precisely, for graphs $F$ and $G$, a
{\em $(k,F)$-sunflower} in $G$ is a $(k+1)$-tuple
$(C,\mathcal F_1,\dots,\mathcal F_k)$, such that $C\subseteq V(G), \mathcal F_i\subseteq \mathcal P(V(G))$,
the sets in $\{C\}\cup\bigcup_i\mathcal F_i$ are pairwise disjoints and
there exists a partition $(K,Y_1,\dots,Y_k)$ of $V(F)$ so that
\begin{itemize}
  \item $\forall i\neq j,\ \omega(Y_i,Y_j)=\emptyset$,
	\item $G[C]\approx F[K]$,
	\item $\forall X_i\in\mathcal F_i, G[X_i]\approx F[Y_i]$,
	\item $\forall (X_1,\dots,X_k)\in\mathcal F_1\times\dots\times\mathcal F_k$, the subgraph  of $G$ induced by $C\cup X_1\cup\dots\cup X_k$ is isomorphic to $F$.
\end{itemize}

\begin{figure}[ht]
\begin{center}
\includegraphics[width=.75\textwidth]{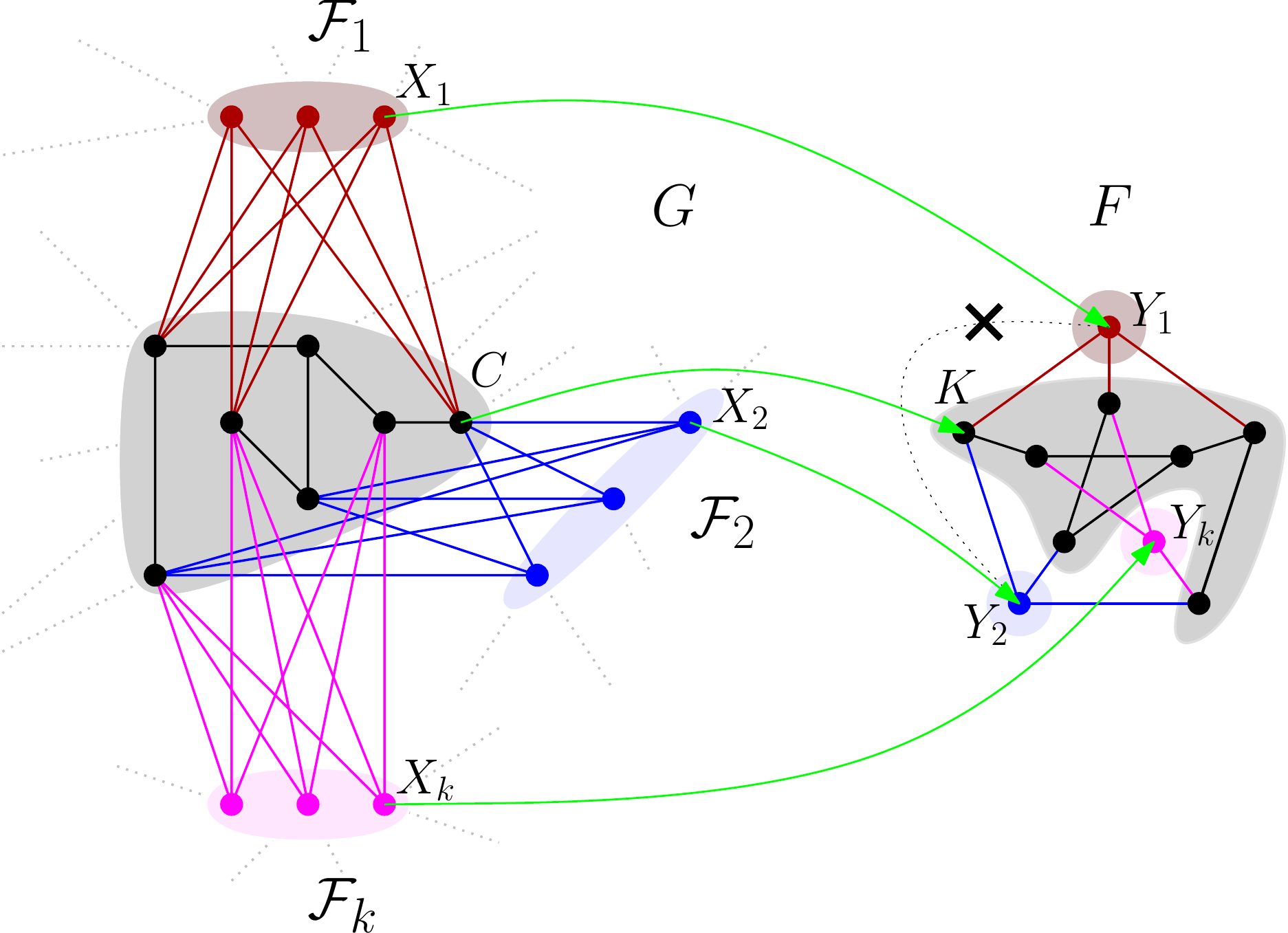}
\caption{A $(3,{\rm Petersen})$-sunflower}
\label{fig:sunflower}
\end{center}
\end{figure}

The following stepping up lemma gives some indication on how low tree-depth decomposition is related to the proof of Theorem~\ref{thm:countF}:

\begin{lemma}[\cite{Taxi_hom}]
There exists a function $\tau$ such that for every integers $p,k$, every graph
 $F$  of order $p$, every $0<\epsilon<1$, the following property holds:

Every graph $G$ such that $(\#F\subseteq G)\,> |G|^{k+\epsilon}$ contains a $({k+1},F)$-sunflower
$(C,\mathcal F_1,\dots,\mathcal F_{k+1})$ with
$$
\min_i |\mathcal F_i|\geq \left(\frac{|G|}{\binom{\chi_p(G)}{p}^{1/\epsilon}}\right)^{\tau(\epsilon,p)}
$$

In particular, $G$ contains a subgraph $G'$ such that
\begin{align*}
&|G'|\geq (k+1)\left(\frac{{|G|}}{{\binom{\chi_p(G)}{p}^{1/\epsilon}}}\right)^{\tau(\epsilon,p)}\\
\text{and}\qquad&(\#F\subseteq G')\geq \left(\frac{|G'|-|F|}{k+1}\right)^{{k+1}}.
\end{align*}
\end{lemma}
\providecommand{\noopsort}[1]{}\providecommand{\noopsort}[1]{}
\providecommand{\bysame}{\leavevmode\hbox to3em{\hrulefill}\thinspace}
\providecommand{\MR}{\relax\ifhmode\unskip\space\fi MR }
% \MRhref is called by the amsart/book/proc definition of \MR.
\providecommand{\MRhref}[2]{%
  \href{http://www.ams.org/mathscinet-getitem?mr=#1}{#2}
}
\providecommand{\href}[2]{#2}

\end{document}